\documentclass[letterpaper]{article}

\usepackage{amsmath}
\usepackage{amsfonts}
\usepackage{amsthm}
\usepackage{amssymb}
\usepackage{datetime}
\usepackage{color}

\theoremstyle{plain}
\newtheorem{theorem}{Theorem}[section]
\newtheorem{lemma}[theorem]{Lemma}
\newtheorem{corollary}[theorem]{Corollary}
\theoremstyle{definition}
\newtheorem{definition}[theorem]{Definition}

\newtheorem{remark}[theorem]{Remark}
\newtheorem{notation}[theorem]{Notation}

\newcommand{\all}{\hbox{for all}}
\newcommand{\All}{\hbox{For all}}

\newcommand{\funder}{\underline f}
\newcommand{\gr}{\hbox{\rm graph}}
\newcommand{\hyp}{\hbox{\rm hyp}}
\newcommand{\lr}{\Longrightarrow}
\newcommand{\qlr}{\quad\Longrightarrow\quad}
\newcommand{\rl}{\Longleftarrow}

\newcommand{\half}{{\textstyle\frac{1}{2}}}
\newcommand{\infn}{\inf\nolimits}
\newcommand{\limn}{\lim\nolimits}

\newcommand{\RR}{\mathbb R}
\newcommand{\tsum}{\textstyle\sum}

\newcommand{\Cor}{Corollary~\ref}
\newcommand{\Cors}{Corollaries~\ref}

\newcommand{\Lem}{Lemma~\ref}

\newcommand{\Rem}{Remark~\ref}
\newcommand{\Sec}{Section~\ref}

\newcommand{\Thm}{Theorem~\ref}

\title{Bootstrapping the Mazur--Orlicz--K\"onig theorem}

\author{Stephen Simons
\thanks{Department of Mathematics, University of California, Santa Barbara, CA\ 93106-3080, U.S.A. Email: \texttt{stesim38@gmail.com}.}}

\date{}

\begin{document}

\maketitle
\begin{abstract}\noindent
In this paper, we give some extensions of K\"onig's extension of the Mazur--Orlicz theorem.   These extensions include generalizations of a surprising recent result of Sun Chuanfeng, and generalizations to the product of more than two spaces of the ``Hahn--Banach--Lagrange'' theorem.  
\end{abstract}



\section{Introduction}\label{Introduction}
In this paper, all vector spaces are {\em real}.   We shall use the terms {\em sublinear}, {\em linear}, {\em convex}, {\em concave} and {\em affine} in their usual senses.
\par
This paper is about extensions of the Mazur--Orlicz theorem, which first appeared in \cite{MOT}: {\em Let $E$ be a  vector space, $S\colon\ E \to \RR$ be sublinear and $C$ be a nonempty convex subset of $E$.   Then there exists a linear map $L\colon\ E \to \RR$ such that $L \le S$ on $E$ and $\infn_CL = \infn_CS.$}   Early improvements and applications of this result were given, in chronological order, by Sikorski \cite{SIK}, Pt\'ak \cite{PTAK}, K\"onig \cite{KONIGMOT}, Landsberg--Schirotzek \cite{LS} and K\"onig \cite{KONIG}.
\par
By a {\em convex--affine version} of a known result we mean that it corresponds to the known result with the word {\em sublinear} in the hypothesis replaced by {\em convex} and the word {\em linear} in the conclusion replaced by {\em affine}.   It is important to note that a convex--affine version of a known result is not necessarily a {\em generalization} of it because an affine function dominated by a sublinear functional is not necessarily linear.     
Recently, Sun Chuanfeng established the following convex--affine version of the Mazur--Orlicz theorem: {\em Let $E$ be a  vector space, $f\colon\ E \to \RR$ be convex and $C$ be a nonempty convex subset of $E$.   Then there exists an affine map $A\colon\ E \to \RR$ such that $A \le f$ on $E$ and $\infn_CA = \infn_Cf.$}  This seems to be a more difficult result that the original Mazur--Orlicz theorem.
\par
The ``Hahn--Banach--Lagrange'' theorem, an existence theorem for linear functionals on a vector space that generalizes the Mazur--Orlicz theorem, first appeared in \cite{NEWHBT}, and the analysis was refined in \cite{HBL} and \cite{HBM}.    The idea behind this result is to provide a unified and relatively nontechnical framework for treating the main existence theorems for continuous linear functionals in linear and nonlinear functional analysis, convex analysis, Lagrange multiplier theory and minimax theory.   Applications were also given to the theory of monotone multifunctions.   In many cases, the Hahn--Banach--Lagrange Theorem leads to necessary and sufficient conditions instead of the more usual known sufficient conditions for the existence of these functionals, and also leads to sharp\break numerical bounds for the norm of the functional obtained.   We will give an analysis of the Hahn--Banach--Lagrange theorem in \Sec{HBLsec}.
\par
We now give a short outline of the analysis in this paper.   \Sec{MOKsec} contains the generalization due to K\"onig of the Mazur--Orlicz theorem, which has the advantage that effort does not have to be expended to prove that certain sets are convex.   The proof is similar to that of the original Mazur--Orlicz theorem, but somewhat more technical.
\par
If $E$ is a vector space and $f\colon\ E \to \RR$ is convex, we use $f$ to define implicitly a sublinear function $S_f\colon E \times \RR \to [\,0,\infty[\,$, to which we will apply the results of \Sec{MOKsec}.   $S_f$ is defined in \Lem{SFlem}, and its properties are explored in \Lem{SFPROPSlem}.   The main result of this section, \Thm{AFFthm}, gives a method for the construction of affine functions.   
\par
In \Sec{SUNsec}, we discuss the result of Sun Chuanfeng's mentioned above.   \Thm{SUNthm} contains a generalization of this result in which the convex subset is replaced by any subset $Z$ satisfying Eqn. \eqref{SUN1}.   Sun Chuanfeng's original result is established in \Cor{SUNcor}.
\par
The original Hahn--Banach--Lagrange theorem used a convex set, $C$, a sublinear functional $S$, and two functions, $j$ and $k$.   See \Cor{HBLcor} for the simplest formulation of this kind of result.  \Cor{HBL} contains a version in which the convex set $C$ is replaces by a set $Z$ with no algebraic structure.   These results are discussed here as consequences of \Cor{ndimaffine} and \Thm{ndim}, which are results on $n\ (\ge 2)$ vector spaces instead of just 2.   \Thm{ndim} is obtained by a very simple bootstrapping procedure from \Lem{MOKlem}.
\par
The question now arises whether there are convex-affine results in the spirit of Sun Chuanfeng's theorem that are similar to \Cor{HBL} and \Cor{HBLcor}.   We give two such results, in \Thm{CAHBLthm} and \Cor{CAHBLcor}.   It would be nice if there were a result analogous to \Thm{CAHBLthm} or \Cor{CAHBLcor} for $n\ (\ge 3)$ vector spaces.   We explain in \Rem{NOGENrem} why we think that this is unlikely.
\par
The author would like to thank Sun Chuanfeng for sending him a preprint of \cite{SUN}. 
\section{On the existence of linear functionals}\label{MOKsec}
\begin{lemma}[Mazur--Orlicz--K\"onig theorem]\label{MOKlem}
Let $E$ be a nonzero vector space, $S\colon\ E \to \RR$ be sublinear and $D$ be a nonempty subset of $E$ such that
$$\all\ d_1,d_2 \in D,\hbox{ there exists }d \in D\hbox{ such that }S\big(d - \half d_1 - \half d_2\big) \le 0.$$
Then there exists a linear map $L\colon\ E \to \RR$ such that $L \le S$ on $E$ and
$$\infn_DL = \infn_DS.$$
\end{lemma}
\begin{proof}
See K\"onig, \cite[Basic Theorem, p.\ 583]{KONIGMOT}.
\end{proof}
\section{An implicitly defined sublinear functional}\label{Implicitly}
\begin{notation}\label{SFnot}
We introduce some notation to simplify the expressions in what follows.   We suppose that $E$ is a  vector space and $f\colon\ E \to \RR$ is convex.   Let $\funder:= f - f(0) - 1$, so that
\begin{equation}\label{not1}
\funder\hbox{ is convex and }\funder(0) = -1.
\end{equation}
If $x \in E$, let $f_x\colon\ \,]0,\infty[\,\to \RR$ be defined by $f_x(\mu) := \mu \funder(x/\mu)$.   We know from \cite[Theorem~2.1.5(v), pp.\ 46--49]{ZBOOK}  that $\funder$ is continuous on any one--dimensional subspace of $E$, and so $f_x$ is continuous.   If $0 < \mu < \nu$ then $0 < \mu/\nu < 1$, and so   
\begin{align*}
\funder(x/\nu) &=  \funder\big((1 - \mu/\nu)0 + (\mu/\nu)(x/\mu)\big)\\
&\le (1 - \mu/\nu)(-1) + (\mu/\nu)\funder(x/\mu) = \mu/\nu - 1 + (\mu/\nu)\funder(x/\mu).
\end{align*}
Multiplying by $\nu$, we see that
\begin{equation}\label{not4}
0 < \mu < \nu \qlr f_x(\nu) + \nu \le f_x(\mu) + \mu. 
\end{equation}
Consequently,   
\begin{equation}\label{not3}
f_x\hbox{ is continuous and strictly decreasing, and }\limn_{\nu \to \infty}f_x(\nu) = -\infty.
\end{equation}   
\end{notation}
\begin{lemma}\label{SFlem}
Let $(x,\alpha) \in E \times \RR$ and $I(x,\alpha) := \{\mu > 0\colon\ f_x(\mu) < \alpha\}$.  Then:
\par\noindent
{\rm(a)} $I(x,\alpha)$ is a semi--infinite open subinterval of $\,]0,\infty[\,$.
\par\noindent
{\rm(b)} If $\inf I(x,\alpha) > 0$ then $\inf I(x,\alpha)$ is the unique value of $\sigma$ with $f_x(\sigma) = \alpha$.
\par\noindent
{\rm(c)} We define the function $S_f\colon E \times \RR \to [\,0,\infty[\,$ by $S_f(x,\alpha) := \inf I(x,\alpha)$.   Then
\begin{equation}\label{Negative}
S_f(x,\alpha) \le 0\iff\all\ \rho \ge 0,\ f(\rho x) - \alpha\rho \le f(0).
\end{equation}
If $S_f(x,\alpha) > 0$ then $S_f(x,\alpha)$ is uniquely determined by the implicit equality
\begin{equation}\label{Implicit}
S_f(x,\alpha)\funder\big(x/S_f(x,\alpha)\big) = \alpha\hbox{\rm\ \big(or equivalently, }(x,\alpha)/S_f(x,\alpha) \in \gr \funder\big).
\end{equation}
\end{lemma}
\begin{proof}
(a)\enspace This follows from \eqref{not3}.
\par
(b)\enspace By hypothesis, there exists $\nu$ such that $0 < \nu < \inf I(x,\alpha)$, from which $f_x(\nu) \ge \alpha$.   From the intermediate value theorem and \eqref{not3}, there exists a unique $\sigma > 0$ such that $f_x(\sigma) = \alpha$.   Now if $\mu \in I(x,\alpha)$ then $f_x(\mu) < \alpha = f_x(\sigma)$, and so $\mu > \sigma$.   Consequently, $\inf I(x,\alpha) \ge \sigma$.   On the other hand, for all $n \ge 1$, $\sigma + 1/n > \sigma$, hence $f_x(\sigma + 1/n) < f_x(\sigma) = \alpha$, and so $\sigma + 1/n \in I(x,\alpha)$.   Consequently, $\inf I(x,\alpha) \le \sigma$.   Thus $\inf I(x,\alpha) = \sigma$, as required.
\par
(c)\enspace We now establish \eqref{Negative}.   Since $f(0x) - \alpha0 = f(0)$, putting $\nu = 1/\rho$,
\begin{gather*}
\all\ \rho \ge 0,\ f(\rho x) - \alpha\rho \le f(0) \iff \all\ \rho > 0,\ f(\rho x) - \alpha\rho \le f(0)\\
\iff \all\ \nu > 0,\ f(x/\nu) - \alpha/\nu \le f(0) \iff  \all\ \nu > 0,\ f_x(\nu) + \nu \le \alpha.\\
\hbox{and, further,}\\
S_f(x,\alpha) \le 0 \iff \all\ \mu > 0,\ \mu \in I(x,\alpha)
\iff \all\ \mu > 0,\ f_x(\mu) < \alpha.
\end{gather*}
It is clear by comparing the two sets of implications above that ``$\rl$'' in \eqref{Negative} is satisfied.   If, on the other hand, for all $\mu > 0$, $f_x(\mu) < \alpha$ and $\nu > 0$, we take\break $0 < \mu < \nu$.   From \eqref{not4}, $f_x(\nu) + \nu \le f_x(\mu) + \mu < \alpha + \mu$.   Letting $\mu \to 0$, we see that $f_x(\nu) + \nu \le \alpha$, and so ``$\lr$'' in \eqref{Negative} is also satisfied.   This gives \eqref{Negative}, and \eqref{Implicit} is is immediate from (b).   
\end{proof}
\begin{definition}
If $\alpha \in \RR$ then $\alpha^-$ is the ``negative part'' of $\alpha$, that is to say, $\alpha^- = \max(-\alpha,0)$.   We write $\hyp\,\funder$ for the ``hypograph'' of $\funder$, that is to say the set $\big\{(x,\alpha) \in E \times \RR\colon\ \funder(x) \ge \alpha\big\}$.   
\end{definition}
\begin{lemma}[Some properties of $S_f$]\label{SFPROPSlem}
We first give the values of $S_f$ on the\break ``vertical axis'', the hypograph of $\funder$ and the graph of $\funder$:
\begin{equation}\label{Axis}
\All\ \alpha \in \RR,\ S_f(0,\alpha) = \alpha^-.
\end{equation}
\begin{equation}\label{Hypograph}
\All\ (x,\alpha) \in \hyp\,\funder,\ S_f(x,\alpha) \ge 1.
\end{equation}
\begin{equation}\label{Graph}
\All\ x \in E,\ S_f\big(x,\funder(x)\big) = 1.
\end{equation}
We next prove that $S_f$ is {\em sublinear}, that is to say
\begin{gather}
S_f(0,0) = 0.\label{S(0)}\\
(x,\alpha) \in E \times \RR\hbox{ and }\lambda > 0 \qlr  S_f(\lambda x,\lambda\alpha) = \lambda S_f(x,\alpha).\label{Sposhom}\\
\noalign{\noindent and}
(x,\alpha)\hbox{ and }(y,\gamma) \in E \times \RR \lr  S_f(x,\alpha) +  S_f(y,\gamma) \ge  S_f(x + y,\alpha + \gamma).\label{Ssubadd}
\end{gather}
\end{lemma}
\begin{proof}
\eqref{Axis} follows from the observation that, for all $\mu > 0$, $f_0(\mu) := -\mu$, and so $I(0,\alpha) = \{\mu > 0\colon\ -\mu < \alpha\} = \,]\alpha^-,\infty[\,$.
\par
If $S_f(x,\alpha) < 1$ then $1 \in I(x,\alpha)$ and so $\funder(x) = f_x(1) < \alpha$.   \eqref{Hypograph} follows from this.
\par
\eqref{Graph} is immediate from \eqref{Hypograph} and the ``uniqueness'' in  \eqref{Implicit}.   
\par
\eqref{S(0)} is immediate from \eqref{Axis}.
\par
From \Lem{SFlem}(b), $I(\lambda x,\lambda\alpha) := \{\mu > 0\colon\ \funder(\lambda x/\mu) < \lambda\alpha/\mu\}$.   Setting $\nu = \mu/\lambda$, so that $\mu = \lambda\nu$, $I(\lambda x,\lambda\alpha) := \{\lambda\nu\colon \nu > 0,\ \funder(x/\nu) < \alpha/\nu\} = \lambda I(x,\alpha)$.   \eqref{Sposhom} follows by taking the infima of both sides.
\par
Let $\mu \in I(x,\alpha)$ and $\nu \in I(y,\gamma)$.   Then $\mu \funder(x/\mu) < \alpha$ and $\nu \funder(y/\nu) < \gamma$.   Thus
$$\funder\bigg(\frac{x + y}{\mu + \nu}\bigg) = \funder\bigg(\frac{\mu(x/\mu) + \nu(y/\nu)}{\mu + \nu}\bigg) \le \frac\mu{\mu + \nu}\funder\bigg(\frac{x}{\mu}\bigg) + \frac\nu{\mu + \nu}\funder\bigg(\frac{y}{\nu}\bigg)$$
$$< \frac\mu{\mu + \nu}\frac{\alpha}{\mu} + \frac\nu{\mu + \nu}\frac{\gamma}{\nu} = \frac{\alpha + \gamma}{\mu + \nu}.$$
Consequently, $\mu + \nu \in I(x + y,\alpha + \gamma)$, and so $\mu + \nu \ge S_f(x + y,\alpha + \gamma)$.   \eqref{Ssubadd} now follows by taking the infimum over $\mu$ and $\nu$.
\end{proof}
We now come to the main result of this section.
\begin{theorem}[The existence of affine maps]\label{AFFthm}
Let $\emptyset \ne B \subset E \times \RR$,
\begin{equation}\label{AFF1}
\infn_{(b,\beta) \in B}\big[f(b) + \beta\big] \in \RR,
\end{equation}
and, for all $(b_1,\beta_1),(b_2,\beta_2) \in B$, there exists $(b,\beta) \in B$ such that
\begin{equation}\label{AFF3}
\all\ \rho \ge 0,\ f\big(\rho\big[b - \half b_1 - \half b_2\big]\big) + \rho\big[\beta - \half\beta_1 - \half\beta_2\big] \le f(0).
\end{equation}
Then there exists an affine map $A\colon\ E \to \RR$ such that $A \le f$ on $E$ and
$$\infn_{(b,\beta) \in B}\big[A(b) + \beta\big] = \infn_{(b,\beta) \in B}\big[f(b)  + \beta\big].$$
\end{theorem}
\begin{proof}
Let $\delta := \infn_{(b,\beta) \in B}\big[f(b) + \beta\big] \in \RR$.
For all $(b,\beta) \in B$, $f(b) + \beta \ge \delta$, and so $(b,\delta - \beta - f(0) - 1) \in \hyp\funder$.   Let $D := \big\{(b,\delta - \beta - f(0) - 1)\big\}_{(b,\beta) \in B}$.   Thus $D \subset \hyp\funder$, and so \eqref{Hypograph} implies that $\inf_DS_f \ge 1$.   From \eqref{AFF3} and \eqref{Negative},
for all $(b_1,\beta_1),(b_2,\beta_2) \in B$, there exists $(b,\beta) \in B$ such that
$$S_f\big(b - \half b_1 - \half b_2,\half\beta_1 + \half\beta_2 - \beta\big) \le 0,$$
which implies that
$$\all\ d_1,d_2 \in D,\hbox{ there exists }d \in D\hbox{ such that }S_f\big(d - \half d_1 - \half d_2\big) \le 0.$$
From \Lem{MOKlem}, with $S = S_f$ and $E$ replaced by $E \times \RR$, there exists a linear map $L\colon\ E \times \RR \to \RR$ such that $L \le S_f$ on $E \times \RR$ and $\infn_DL = \infn_DS_f \ge 1$.   From algebraic considerations, there exist a linear map $\Lambda\colon E \to \RR$ and $\lambda  \in \RR$ such that, for all $(x,\alpha) \in E \times \RR$, $L(x,\alpha) = \Lambda(x) - \lambda\alpha$.   Thus, from \eqref{Graph},     
\begin{gather}
\all\ x \in E,\quad \Lambda(x) - \lambda \funder(x) = L\big(x,\funder(x)\big) \le S_f\big(x,\funder(x)\big) = 1,\label{AFF4}\\
\noalign{\noindent and, for all $(b,\eta) \in D$,\quad$\Lambda(b) - \lambda\eta = L\big(b,\eta\big) \ge \infn_DL \ge 1$, from which}
\all\ (b,\beta) \in B,\quad \Lambda(b) - 1 + \lambda(f(0) + 1 + \beta) \ge \lambda\delta.\label{AFF5}
\end{gather}
From \eqref{Axis}, $-\lambda = \Lambda(0) - \lambda 1 = L(0,1) \le S_f(0,1) = 0$, so $\lambda \ge 0$.   Now, if we had $\lambda = 0$ then, from \eqref{AFF4}, for all $x \in E$, $\Lambda(x) \le 1$.   Consequently, $\Lambda = 0$, which would contradict \eqref{AFF5}.   Thus $\lambda > 0$.   We write $A$ for the affine function $\Lambda/\lambda - 1/\lambda  + f(0) + 1$.   If we divide \eqref{AFF4} by $\lambda$ and rearrange the terms we see that, for all $x \in E$, $(\Lambda/\lambda)(x) - 1/\lambda \le \funder(x)$, from which $A(x) \le \funder(x) + f(0) + 1 = f(x)$.   Thus $A \le f$ on $E$.   Consequently,
\begin{equation}\label{AFF6}
\infn_{(b,\beta) \in B}\big[A(b) + \beta\big] \le \infn_{(b,\beta) \in B}\big[f(b)  + \beta\big].
\end{equation} 
Dividing \eqref{AFF5} by $\lambda$, for all $(b,\beta) \in B$, $(\Lambda/\lambda)(b) - 1/\lambda + f(0) + 1 + \beta \ge \delta$.   Thus $A(b) + \beta \ge \delta$.   Consequently,
$$\infn_{(b,\beta) \in B}\big[A(b) + \beta\big] \ge \delta = \infn_{(b,\beta) \in B}\big[f(b) + \beta\big].$$
The required result follows by combining this with \eqref{AFF6}.
\end{proof}
\begin{remark}
Another way of seeing the sublinearity of $S_f$ is to note that $S_f$ is the {\em Minkowski functional} of the strict epigraph of $\funder$.
\end{remark}
\section{Sun Chuanfeng's theorem}\label{SUNsec}
\begin{lemma}\label{Affine}
Let $E$ be a  vector space, $f\colon\ E \to \RR$ be convex and $x \in E$.   Then there exists an affine map $A\colon\ E \to \RR$ such that $A \le f$ on $E$ and $A(x) = f(x)$.   
\end{lemma}
\begin{proof}
Let $B = \big\{(x,0)\big\}$.   Then $\infn_{(b,\beta) \in B}\big[f(b) + \beta\big] = f(x)$ and, for all $\rho \ge 0$, $f\big(\rho[0]\big) + \rho[0] = f(0).$  
The result follows from \Thm{AFFthm}.      
\end{proof}
\begin{remark}
\Lem{Affine} says that the algebraic subdifferential of $f$ at $x$ is nonempty.   This can also be deduced by applying the Hahn--Banach theorem to the sublinear functional introduced in \cite[Theorem~2.1.13, pp.\ 55--56]{ZBOOK}.
\end{remark}
\begin{theorem}[A generalization of Sun Chuanfeng's theorem]\label{SUNthm}
Let $E$ be a  vector space, $f\colon\ E \to \RR$ be convex, $\emptyset \ne Z\subset E$ and, for all $z_1,z_2 \in Z$, there exists $z \in Z$ such that  
\begin{equation}\label{SUN1}
\all\ \rho \ge 0,\ f\big(\rho\big[z - \half z_1 - \half z_2\big]\big) \le f(0).\end{equation}
Then there exists an affine map $A\colon\ E \to \RR$ such that $A \le f$ on $E$ and\break $\inf_ZA = \inf_Zf$.   
\end{theorem}
\begin{proof}
If $\inf_Zf = -\infty$ then the result is immediate from \Lem{Affine}, so we can and will suppose that $\inf_Zf \in \RR$.   Let $B := Z \times \{0\}$. \eqref{AFF1} follows since $\infn_{(b,\beta) \in B}\big[f(b) + \beta\big] = \inf_Zf$, and \eqref{AFF3} is immediate from \eqref{SUN1}.   \Thm{AFFthm}\break gives an affine map $A\colon\ E \to \RR$ such that $A \le f$ on $E$ and
$$\infn_{z \in Z}\big[A(z) + 0\big] = \infn_{z \in Z}\big[f(z) + 0\big].$$
The desired result follows since $\infn_{z \in Z}\big[A(z) + 0\big] = \inf_ZA$ and, as we have already observed, $\infn_{(b,\beta) \in B}\big[f(b) + \beta\big] = \inf_Zf$. 
\end{proof}
\begin{corollary}[Sun Chuanfeng's theorem]\label{SUNcor}
Let $E$ be a  vector space, $f\colon\ E \to \RR$ be convex, and $C$ be a nonempty convex {\em(or even midpoint convex)} subset of $E$.   Then there exists an affine function $A\colon\ E \to \RR$ such that $A \le f$ on $E$ and $\inf_CA = \inf_Cf$.   
\end{corollary}
\begin{proof}
This is immediate from \Thm{SUNthm}. 
\end{proof}
\begin{remark}
In a certain sense, the analysis presented in this section is actually only ``half the story'' because $B \subset E \times \{0\}$.   We will discuss the ``whole story'' in \Sec{CAHBLsec}.   Specifically, \Thm{SUNthm} will be extended by \Thm{CAHBLthm}, and \Cor{SUNcor} will be extended by \Cor{CAHBLcor}.  
\end{remark}
\section{Results of Hahn--Banach--Lagrange type}\label{HBLsec}
The main result in this section is \Thm{ndim}, which represents a considerable generalization of the ``Hahn--Banach--Lagrange'' theorem, \cite[Theorem 1.11,\ p. 21]{HBM}.  \Thm{ndim} is an easy consequence of \Lem{MOKlem}.   \Cor{ndimaffine}--\ref{HBLcor} are simple special cases of \Thm{ndim}.   
 See \cite{HBM} for a discussion of the many consequences of \Cors{HBL} and \ref{HBLcor}.
\begin{theorem}\label{ndim}
For all $m = 1,\dots n$, let $E_m$ be a vector space and $S_m\colon\ E_m \to \RR$ be sublinear.   Let $Z$ be a nonempty set, for all $m = 1,\dots n$, $j_m\colon\ Z \to E_m$ and, for all $z_1,z_2 \in Z$, there exists $z \in Z$ such that
$$\tsum_{m = 1}^n S_m\big(j_m(z) - \half j_m(z_1) - \half j_m(z_2)\big) \le 0.$$
Then, for all $m = 1,\dots n$, there exists a linear map $L_m\colon\ E_m \to \RR$ such that $L_m \le S_m$ on $E_m$ and
\abovedisplayskip6pt\belowdisplayskip0pt
$$\infn_Z\tsum_{m = 1}^n\big[L_m\circ j_m\big] = \infn_Z\tsum_{m = 1}^n\big[S_m\circ j_m\big].$$
\end{theorem}
\begin{proof}
Define $j\colon\ Z \to E_1 \times \cdots \times E_n$ by $j(z) := \big(j_1(z),\dots,j_m(z)\big)$, and define the sublinear map $S\colon\ E_1 \times \cdots \times E_n \to \RR$ by $S(x_1,\dots,x_n) := \sum_{m = 1}^n S_m(x_m)$.    Then the conditions of \Lem{MOKlem} are satisfied with $E := E_1 \times \cdots \times E_n$ and $D := j(Z)$.   From \Lem{MOKlem}, there exists a linear map $L\colon\ E \to \RR$ so that $L \le S$ on $E$ and $\infn_DL = \infn_DS$.   For all $m = 1,\dots, n$, there exists a linear map $L_m\colon\ E_m \to \RR$ such that, for all $(x_1,\dots,x_n) \in E_1 \times \cdots \times E_n$, $L(x_1,\dots,x_n) = \tsum_{m = 1}^nL_m(x_m)$.   Now let $1 \le m \le n$, $y \in E_m$ and $w = (0,\dots,y,\dots,0)$, where the ``$y$'' is in the $m^{\rm{th}}$ place.   Then $L_m(y) = L(w) \le S(w) = S_m(y)$, from which $L_m \le S_m$ on $E_m$.  The result follows since $\infn_Z\sum_{m = 1}^n\big[L_m\circ j_m\big] = \infn_ZL\circ j = \infn_DL$ and $\infn_Z\sum_{m = 1}^n\big[S_m\circ j_m\big] = \infn_ZS\circ j = \infn_DS$. 
\end{proof}
\begin{corollary}\label{ndimaffine}
For all $m = 1,\dots n$, let $E_m$ be a vector space and $S_m\colon\ E_m \to \RR$ be sublinear.   Let $C$ be a nonempty convex subset of a vector space and, for all $m = 1,\dots n$, $j_m\colon\ C \to E_m$ be affine.   Then, for all $m = 1,\dots n$, there exists a linear map $L_m\colon\ E_m \to \RR$ such that $L_m \le S_m$ on $E_m$ and
\abovedisplayskip6pt\belowdisplayskip0pt
$$\infn_C\tsum_{m = 1}^n\big[L_m\circ j_m\big] = \infn_C\tsum_{m = 1}^n\big[S_m\circ j_m\big].$$
\end{corollary}
\begin{proof}
This is immediate from \Thm{ndim}.
\end{proof}
The original ``Hahn--Banach--Lagrange'' theorem appeared in (among other places)  \cite[Theorem 1.11,\ p. 21]{HBM}.   \Cor{HBL} below is a generalization of the generalization of this that appeared in \cite[Theorem 1.13,\ p. 22]{HBM}.    
\begin{corollary}\label{HBL}
Let $E$ be a vector space and $S\colon\ E \to \RR$ be sublinear.   Let $Z \ne \emptyset$,  $j\colon\ Z \to E$, $k\colon Z \to \RR$ and, whenever $z_1,z_2 \in Z$, there exists $z \in Z$ such that
$$S\big(j(z) - \half j(z_1) - \half j(z_2)\big) + k(z) - \half k(z_1) - \half k(z_2) \le 0.$$
Then there exists a linear map $L\colon\ E \to \RR$ such that $L \le S$ on $E$ and
\abovedisplayskip6pt\belowdisplayskip0pt
$$\infn_Z\big[L\circ j + k\big] = \infn_Z\big[S\circ j + k\big].$$
\end{corollary}
\begin{proof}
We apply \Thm{ndim} with $n = 2$, $E_1 = E$, $E_2 = \RR$, $S_1 = S$, $S_2 = I_\RR$, the identity map on $\RR$, $j_1 = j$ and $j_2 = k$.   Thus there exist a linear map $L\colon\ E \to \RR$ and a linear map $M\colon\ \RR \to \RR$ such that $L \le S$ on $E$, $M \le I_\RR$ on $\RR$ and $\infn_Z\big[L\circ j + M\circ k\big] = \infn_Z\big[S\circ j + M\circ k\big]$.  The result follows since, as is easily seen, $M = I_\RR$.    
\end{proof}
We have included \Cor{HBLcor} below so that it can be compared with the convex--affine result contained in \Cor{CAHBLcor}.
\begin{corollary}\label{HBLcor}
Let $E$ be a vector space and $S\colon\ E \to \RR$ be sublinear.   Let $C$ be a nonempty convex subset of a vector space,  $j\colon\ C \to E$ be affine and $k\colon C \to \RR$ be convex.   Then there exists a linear map $L\colon\ E \to \RR$ such that $L \le S$ on $E$ and
\abovedisplayskip0pt\belowdisplayskip0pt
$$\infn_C\big[L\circ j + k\big] = \infn_C\big[S\circ j + k\big].$$
\end{corollary}
\begin{proof}
This is immediate from \Cor{HBL}.
\end{proof}
\section{A convex--affine result of Hahn--Banach--\\Lagrange type}\label{CAHBLsec}
Our next result is a convex--affine version of \Cor{HBL}, and its consequence, \Cor{CAHBLcor}, is a convex--affine version of \Cor{HBLcor}.  
\begin{theorem}[A convex--affine result of Hahn--Banach--Lagrange type]\label{CAHBLthm}
Let $E$ be a vector space, $f\colon\ E \to \RR$ be convex, $Z \ne \emptyset$, $j\colon Z \to E$, $k\colon Z \to \RR$ and, whenever $z_1,z_2 \in Z$, there exists $z \in Z$ such that, for all $\rho \ge 0$, 
\begin{equation}\label{CAHBL1}
f\big(\rho\big[j(z) - \half j(z_1) - \half j(z_2)\big]\big) + \rho\big[k(z) -\half k(z_1) - \half k(z_2)\big] \le f(0).
\end{equation} 
Then there exists an affine map $A\colon\ E \to \RR$ such that $A \le f$ on $E$ and
\abovedisplayskip3pt\belowdisplayskip0pt
$$\infn_Z\big[A\circ j + k\big] = \infn_Z\big[f\circ j + k\big].$$   
\end{theorem}
\begin{proof}
If $\infn_Z\big[f\circ j + k\big] = -\infty$ then the result is immediate from \Lem{Affine}, so we can and will suppose that $\infn_Z\big[f\circ j + k\big] \in \RR$.   Let $B = \big\{\big(j(z),k(z)\big)\big\}_{z \in Z}$.   \eqref{AFF1} now follows since $\infn_{(b,\beta) \in B}\big[f(b) + \beta\big] = \infn_Z\big[f\circ j + k\big]$, and \eqref{AFF3} is immediate from \eqref{CAHBL1}. \Thm{AFFthm} now provides an affine map $A\colon\ E \to \RR$ such that $A \le f$ on $E$ and $\infn_{(b,\beta) \in B}\big[A(b) + \beta\big] = \infn_{(b,\beta) \in B}\big[f(b)  + \beta\big]$.   The desired result follows since $\infn_{(b,\beta) \in B}\big[A(b) + \beta\big] = \infn_Z\big[A\circ j + k\big]$ and, as was noted above, $\infn_{(b,\beta) \in B}\big[f(b) + \beta\big] = \infn_Z\big[f\circ j + k\big]$.
\end{proof}
\begin{corollary}\label{CAHBLcor}
Let $E$ be a vector space, $f\colon\ E \to \RR$ be convex, $C$ be a nonempty convex subset of a vector space, $j\colon C \to E$ be affine, and $k\colon C \to \RR$ be convex.   Then there exists an affine map $A\colon\ E \to \RR$ such that $A \le f$ on $E$ and
\abovedisplayskip3pt\belowdisplayskip0pt
$$\infn_C\big[A\circ j + k\big] = \infn_C\big[f\circ j + k\big].$$   
\end{corollary}
\begin{proof}
This is immediate from \Thm{CAHBLthm}.
\end{proof}
\begin{remark}\label{NOGENrem}
It is tempting to try to find an analog of \Thm{CAHBLthm} for $n \ge 3$ convex functions instead of 2 in the general spirit of \Thm{ndim}.   The problem is that the technique used in the proof of \Thm{ndim} of progressively setting all the values of $x_m$ other that one particular one to be 0 and using the fact that linear and sublinear maps vanish at 0 does not seem available in the convex--affine case.     
\end{remark}


%

\begin{thebibliography}{50}
%
\bibitem{KONIGMOT}H. K\"onig, {\em On certain applications of the
Hahn--Banach and minimax theorems}, Arch. Math. {\bf 21} (1970),
583--591. 
\bibitem{KONIG}H. K\"onig, {\em Some Basic Theorems in Convex
Analysis}, in ``Optimization and operations research'', edited
by B. Korte, North-Holland (1982).
%
\bibitem{LS}M. Landsberg, W. Schirotzek, {\em Mazur--Orlicz type theorems with applications}, Math. Nachr. {\bf 79} (1977), 331--341. 
%
\bibitem{PTAK}V. Pt\'ak, {\em On a theorem of Mazur and Orlicz},  Studia Math. {\bf 15} (1956), 365--366.
%
\bibitem{MOT}S. Mazur, W. Orlicz, {\em Sur les espaces m\'etriques linŽaires II}, Studia Math. {\bf 13} (1953), 137--179.
%
\bibitem{SIK}R. Sikorski, {\em On a theorem of Mazur and Orlicz},  Studia Math. {\bf 13} (1953), 180--182.
%
\bibitem{NEWHBT}S. Simons, {\em A new version of the
Hahn--Banach theorem}, Arch. Math. {\bf 80} (2003), 630--646.
%
\bibitem{HBL}------, {\em The Hahn--Banach--Lagrange theorem},
Optimization {\bf 56} (2007), 149--169.
%
\bibitem{HBM}-----, {\em From Hahn--Banach to monotonicity}, 
Lecture Notes in Mathematics, {\bf 1693}, second edition, (2008), Springer--Verlag.
%
\bibitem{SUN}Sun Chuanfeng, {\em The Mazur--Orlicz theorem for convex functionals}, J. of Convex Anal., to appear.  (Personal communication.)  
%
\bibitem{ZBOOK} C. Z\u{a}linescu, {\em Convex analysis in general vector spaces}, (2002), World Scientific.
\end{thebibliography}
\end{document}